\documentclass[11pt]{amsart}

\usepackage{amsmath,amsfonts,graphicx,amsthm,mathdots,url,listings}
\usepackage[foot]{amsaddr}

\newcommand{\R}{\mathbb{R}}
\newcommand{\floor}[1]{\left\lfloor {#1} \right\rfloor}
\newcommand{\ceil}[1]{\left\lceil {#1} \right\rceil}

\newtheorem{thm}{Theorem}[section]
\newtheorem{lem}{Lemma}[section]

\DeclareMathOperator{\HJ}{HJ}
\DeclareMathOperator{\TTT}{TTT}
\DeclareMathOperator{\Gr}{Graham}

\newcommand{\img}[1]{\begingroup
\setbox0=\hbox{\includegraphics{#1}}%
\parbox{\wd0}{\box0}\endgroup}

\title{Graham's number is less than $2\uparrow \uparrow \uparrow 6$}
\author{Mikhail Lavrov$^1$}
\address{$^1$Carnegie Mellon University, Mathematics Department}
\author{Mitchell Lee$^2$}
\address{$^2$Massachusetts Institute of Technology}
\author{John Mackey$^3$}
\address{$^3$Carnegie Mellon University, Mathematics Department}

\email{mlavrov@andrew.cmu.edu, mitchlee@mit.edu, jmackey@andrew.cmu.edu}

\begin{document}

\begin{abstract}
In \cite{graham1971ramsey}, Graham and Rothschild consider a geometric Ramsey problem: finding the least $n$ such that if all edges of the complete graph on the points $\{\pm1\}^n$ are $2$-colored, there exist $4$ coplanar points such that the $6$ edges between them are monochromatic. They give an explicit upper bound: $n \le F(F(F(F(F(F(F(12)))))))$, where $F(m) = 2 \uparrow^{m} 3$, an extremely fast-growing function. We bound $n$ between two instances of a variant of the Hales-Jewett problem, obtaining an upper bound which is between $F(4)$ and $F(5)$.
\end{abstract}

\maketitle

\section{Introduction}

\subsection{Preliminaries}

Let $A$ be a finite set with at least $2$ elements, and $A^n$ the set of $n$-tuples of elements of $A$. Fix a group $G$ which acts on $A$. We define a $k$-parameter subset of $A^n$ to be the image of an injection $f: A^k \to A^n$ which has a specific form: for all $1 \le i \le n$, either
\begin{enumerate}
\item $f_i(x_1, \dots, x_k) = a$ for some $a \in A$, or
\item $f_i(x_1, \dots, x_k) = \sigma(x_j)$ for some $1 \le j \le k$ and some $\sigma \in G$.
\end{enumerate}
The requirement that $f$ is an injection is equivalent to asking that for all $1 \le j \le k$, there exist $i, \sigma$ such that $f_i(x_1, \dots, x_k) = \sigma(x_j)$. Notably, if $f$ defines a $t$-parameter subset of $A^n$, and $g$ defines a $k$-parameter subset of $A^t$, then $f \circ g$ defines a $k$-parameter subset of $A^n$.

The $n$-parameter sets were introduced by Graham and Rothschild \cite{graham1971ramsey}, who proved the following result:

\begin{thm}[Graham-Rothschild Parameter Sets Theorem]
\label{thm:g-r}
Pick an alphabet $A$, a group $G$ acting on $A$, and integers $0 \le k \le t$ and $c \ge 2$. Then there exists a $N$ such that for all $n \ge N$, if the $k$-parameter subsets of $A^n$ are colored one of $c$ colors, then some $t$-parameter subset of $A^n$ can be found, all of whose $k$-parameter subsets receive the same color.
\end{thm}

For the remainder of this paper, we will only consider $2$-colorings ($c=2$) and, when necessary, we will call these two colors ``red" and ``blue".

There are several special cases of Theorem~\ref{thm:g-r} which are of interest.

\subsubsection{Graham's number}

Take $A = \{\pm1\}$, and let $G$ be the group of both permutations of $A$: $\{x \mapsto x, x \mapsto -x\}$. Then any two points in $\{\pm1\}^n$ form a $1$-parameter set. More generally, a $d$-parameter set consists of $2^d$ points that lie on a $d$-dimensional affine subspace of $\R^n$ (if we include $\{\pm1\}^n \subset \R^n$ in the natural way). We will also call this a $d$-dimensional subcube of $\{\pm1\}^n$.

An edge-coloring of $\{\pm1\}^n$ is a $2$-coloring of the edges of the complete graph on the $2^n$ points of $\{\pm1\}^n$: a coloring of the $1$-parameter subsets of $\{\pm1\}^n$. Let $\Gr(d)$ be the smallest dimension $n$ such that every edge-coloring of the $n$-dimensional cube contains a monochromatic $d$-dimensional subcube. Then Theorem~\ref{thm:g-r} implies that $\Gr(d)$ exists and is finite for all $d$.

In particular, $\Gr(2)$ is the smallest dimension $n$ such that every edge-coloring of $\{\pm1\}^n$ contains a monochromatic planar $K_4$: a set of $4$ coplanar points in $\{\pm1\}^n$ such that all $6$ edges between them are the same color. An incredibly large upper bound on $\Gr(2)$ was popularized as ``Graham's number" by Martin Gardner \cite{gardner2001colossal}.

\subsubsection{The Hales-Jewett number}

Take $A = [t] := \{1, 2, \dots, t\}$, let $G = \{e\}$, and color the $0$-parameter sets of $[t]^n$, which is equivalent to coloring elements of $[t]^n$. A $1$-parameter set in $[t]^n$ is called a combinatorial line, and a $d$-parameter set is called a $d$-dimensional combinatorial space.

The Hales-Jewett number $\HJ(t,k)$ be the least dimension $n$ such that every $k$-coloring of $[t]^n$ contains a monochromatic combinatorial line. More generally, $\HJ(t,k,d)$ is the least dimension $n$ such that every $k$-coloring of $[t]^n$ contains a monochromatic $d$-dimensional combinatorial space.

The Hales-Jewett number is probably the most well-studied of the three problems. In \cite{shelah1988primitive}, Shelah proved a primitive recursive upper bound on $\HJ(t, k, d)$. 

\subsubsection{The tic-tac-toe number}

Again consider coloring the elements of $[t]^n$, but this time allow a wider variety of $d$-parameter subsets: let $G = \{e, \pi\}$, which acts on $[t]$ by $e(x) = x$ and $\pi(x) = t+1-x$. A $1$-parameter subset using this $G$ is called a tic-tac-toe line, and a $d$-parameter subset a $d$-dimensional tic-tac-toe space. We define the tic-tac-toe numbers $\TTT(t,k)$ and $\TTT(t,k,d)$ analogously to the Hales-Jewett numbers.

Since the same set of points is colored, but more subsets are acceptable, it's clear that $\TTT(t,k,d) \le \HJ(t,k,d)$. Furthermore, it is easily shown that for all $t$, $k$, and $d$, $\HJ(\ceil{t/2}, k, d) \le \TTT(t,k,d)$, so the overall behavior of the tic-tac-toe numbers and Hales-Jewett numbers are similar. However, for small values of $t$ (and we will only consider the case $t=4$) the behavior of these two bounds is potentially quite different, and it is therefore worthwhile to state our results in terms of the tic-tac-toe number instead.

\subsection{Previous results on $\Gr(d)$}

In \cite{graham1971ramsey}, Graham and Rothschild observed that $\Gr(2) \ge 6$, and conjectured that $\Gr(2) < 10$. This conjecture has been proven false, but not by much: the lower bound was later improved to $11$ by Exoo \cite{exoo2003euclidean} and then to $13$ by by Barkley \cite{barkley2008improved}.

The upper bound is a more complicated story. Although most sources list the bound from \cite{graham1971ramsey} as the best upper bound known, this is not the case. In \cite{shelah1988primitive}, Shelah proved a bound on Theorem~\ref{thm:g-r} for $G = \{e\}$, as well as a similar proof of the affine Ramsey theorem. These can be used to obtain primitive recursive bounds on $\Gr(d)$; for example, Theorem 5.1 in \cite{carlson2006infinitary} can be used to prove a bound on $\Gr(2)$ which uses $17$ iterated applications of the Hales-Jewett number.
 
\subsection{New results}

Our primary results are the following improved bounds on $\Gr(d)$ and especially $\Gr(2)$:

\begin{thm}$\TTT(4,2,d) + 1 \le \Gr(d+1)$.
\label{thm:lower}
\end{thm}

\begin{thm}$\Gr(2) \le \TTT(4,2,6) + 1$.
\label{thm:upper}
\end{thm}

In particular, if we bound $\TTT(4,2,6)$ by $\HJ(4,2,6)$, then by Lemma~\ref{lem:426}, which analyzes the growth rate of Shelah's bound on $\HJ(t,k,d)$, we have $\Gr(2) \le 2 \uparrow \uparrow 2 \uparrow \uparrow 2 \uparrow \uparrow 9 < 2 \uparrow\uparrow \uparrow 6$. This is a significant improvement on all previously known bounds.

The proof of Theorem~\ref{thm:upper}, however, is not completely satisfactory. For Lemma~\ref{lem:computer}, which states that under some strong simplifying assumptions a monochromatic $K_4$ exists in dimension $n=6$ (a tight bound), we only have a computer-aided proof. However, a weaker version of this lemma can be easily shown, yielding:

\begin{thm}There exists a positive integer $d$, e.g.\ $d=2 \uparrow \uparrow 18$, such that $\Gr(2) \le \TTT(4,2,d) + 1$.
\label{thm:upperx}
\end{thm}

This is still easily strong enough to yield $2 \uparrow \uparrow \uparrow 6$ as an upper bound.

We also consider a simpler problem: given an edge-coloring $\{\pm1\}^n$, to find a monochromatic planar rectangle. The points defining a rectangle are still a $2$-parameter set; however, rather than requiring that all $6$ edges between them are monochromatic, we only consider the $4$ edges between ``adjacent" points. 

This simplified problem has a much smaller upper bound:

\begin{thm}
An edge-coloring of $\{\pm1\}^{78}$ necessarily contains a monochromatic planar square whose sides have Hamming length $2$.
\label{thm:squares}
\end{thm}

\section{Bounds on $\Gr(d)$}

\subsection{Setup}

Let $Q$ be the cube $\{\pm1\}^{n+1}$ with coordinates numbered $0, \dots, n$. Let $Q^- = \{x \in Q : x_0 = -1\}$ and $Q^+ = \{x \in Q : x_0 = +1\}$. 

Let $\phi : \{\pm1\}^2 \to [4]$ be given by $\phi(-1,-1) = 1$, $\phi(-1,+1) = 2$, $\phi(+1,-1) = 3$, $\phi(+1,+1) = 4$. We use $\phi$ to define a bijection $\Phi$ from the edges $Q^- \times Q^+$ to $[4]^n$:
$$\Phi(x^-, x^+) = (\phi(x^-_1, x^+_1), \dots, \phi(x^-_n, x^+_n)).$$

For a $(d+1)$-dimensional subcube of $Q$, there are three possibilities: either it is contained entirely in $Q^-$, or entirely in $Q^+$, or half of its vertices are in $Q^-$ and half are in $Q^+$. In the third case, we call the subgraph formed by the edges of the subcube going from $Q^-$ to $Q^+$ a \emph{$d$-dimensional hyperbowtie} (the name is formed by analogy with the case $d=1$, in which case the four edges make a bowtie shape).

\begin{lem}
\label{lem:hyperbowtie}
If $S \subseteq Q^- \times Q^+$ is a set of edges of $Q$, then $S$ is a $d$-dimensional hyperbowtie if and only if $\Phi(S)$ is a $d$-dimensional tic-tac-toe subspace of $[4]^n$.
\end{lem}

\begin{proof}
Let $f : \{\pm1\}^{d+1} \to \{\pm1\}^{n+1}$ be a function whose image is a $(d+1)$-dimensional subcube half contained in $Q^-$ and half in $Q^+$. Then $f_0$ cannot be constant; because so far, all coordinates of $\{\pm1\}^{d+1}$ are symmetric, we may assume that $f_0(x_0, \dots, x_d) = x_0$. Define $g : [4]^d \to [4]^n$ as follows: if $(z_1, \dots, z_d) \in [4]^d$, let $(x_i,y_i) = \phi^{-1}(z_i)$ for $1 \le i \le d$, and let $g(z_1, \dots, z_d) = \Phi(f(-1,x_1,\dots,x_d), f(+1,y_1,\dots,y_d))$. As $z$ varies, the edge from $(-1,x)$ to $(+1,y)$ varies over all edges in $S$, the $d$-dimensional hyperbowtie corresponding to the image of $f$. Therefore the image of $g$ is $\Phi(S)$.

For each coordinate $1 \le i \le d$, we consider all possibilities for $f_i$, and check what form $g_i$ then has:
\begin{itemize}
\item If $f_i$ is a constant $\pm1$, then $f_i(-1,x_1, \dots, x_d) = f_i(+1,y_1, \dots,y_d) = \pm1$, and so $g_i(z) = \phi(\pm1, \pm1)$ which is either a constant $1$ or a constant $4$.

\item If $f_i(x) = \pm x_0$, then $f_i(-1, x_1, \dots, x_d)$ and $f_i(+1,y_1, \dots, y_d)$ are independent of $x,y$ and have opposite signs, so $g_i(z) = \phi(\pm1, \mp 1)$ which is either a constant $2$ or a constant $3$.

\item If $f_i(x) = x_j$, for $j \ge 1$, then $g_i(z) = \phi(x_j, y_j) = z_j$.

\item If $f_i(x) = -x_j$, for $j \ge 1$, then $g_i(z) = \phi(-x_j, -y_j)$. It can be checked that $\phi(-x,-y) = 5-\phi(x,y)$, and so $g_i(z) = 5 - \phi(x_j,y_j) = 5-z_j$.
\end{itemize}

Therefore $g$ has the correct form for the image of $g$ to be a $d$-dimensional tic-tac-toe subspace. Moreover, every possibility for $g_i$ can be obtained by some choice of $f_i$, and so every $d$-dimensional tic-tac-toe subspace can be obtained in this way: as the image under $\Phi$ of a $d$-dimensional hyperbowtie.
\end{proof}

\subsection{The lower bound}

\begin{proof}[Proof of Theorem~\ref{thm:lower}]
Let $n = \Gr(d+1) - 1$ and let $Q$ be the cube $\{\pm1\}^{n+1}$. 

Pick an arbitrary $2$-coloring of $[4]^n$. The map $\Phi$ is a bijection between $[4]^n$ and those edges of $Q$ which change the first coordinate, so we use this bijection to assign those edges a color. To color the remaining edges, we assign the edge from $(x_0, x_1, \dots, x_n)$ to $(y_0, y_1, \dots, y_n)$, where $x_0 = y_0$, the same color as the edge from $(-1,x_1, \dots, x_n)$ to $(+1, y_1, \dots, y_n)$.

Because $n+1 = \Gr(d+1)$, a $(d+1)$-dimensional subcube of $Q$ is monochromatic. Suppose this subcube is half contained in $Q^+$ and half in $Q^-$. Then the edges of the subcube contained in $Q^+ \times Q^-$ form a monochromatic $d$-dimensional hyperbowtie, and by Lemma~\ref{lem:hyperbowtie}, $\Phi$ maps it to a monochromatic $d$-dimensional tic-tac-toe space in $[4]^n$.

Now consider the other possibility: the subcube is entirely contained in $Q^+$ or $Q^-$. Let $i$ be the first coordinate which is not constant on this subcube. We restrict our attention to the $4^d$ edges in the subcube which change coordinate $i$: edges from $(x_0, \dots, x_{i-1}, -1, x_{i+1}, \dots, x_n)$ to $(y_0, \dots, y_{i-1}, +1, y_{i+1}, \dots, y_n)$, where $x_0 = y_0$, $x_1 = y_1$, \dots, $x_{i-1} = y_{i-1}$. Alter each edge by replacing $x_0$ with $-1$ and $y_0$ with $+1$. By construction, the new edge has the same color, so the edges we obtain will also be monochromatic. But now the edges we get form a $d$-dimensional hyperbowtie, and we use Lemma~\ref{lem:hyperbowtie} again to obtain a monochromatic $d$-dimensional tic-tac-toe space in $[4]^n$.

Therefore we have shown that $[4]^n$ always contains a monochromatic $d$-dimensional tic-tac-toe space, and so $n \ge \TTT(4,2,d)$.
\end{proof}

\subsection{A special case}

To prove the upper bound on $\Gr(2)$, we will first state a lemma about a special case of this problem.

\begin{lem}
\label{lem:computer}
Suppose the cube $\{\pm1\}^6$ is $2$-colored so that all parallel edges receive the same color. Then the cube contains a monochromatic planar $K_4$.
\end{lem}

Unfortunately, we do not have a proof of this lemma. However, with the parallel edge assumption, the coloring problem can be formulated as a SAT instance with $364$ variables; a computerized search showed that no solutions exist.

It is possible, however, to prove a weaker version of the lemma.

\begin{lem}
\label{lem:uncomputer}
Let $n = 2 \uparrow \uparrow 18$, and suppose the cube $\{\pm1\}^n$ is $2$-colored so that all parallel edges receive the same color. Then the cube contains a monochromatic planar $K_4$.
\end{lem}

\begin{proof}
An equivalence class of parallel edges in the cube $\{\pm1\}^n$ can be described by a \emph{direction} $a \in \{-1,0,1\}^n$, corresponding to all possible edges $(x, x+a)$; $a$ and $-a$ represent the same direction. 

We will define addition and subtraction of directions componentwise. In order for $a+b$ and $a-b$ to also be directions, we require that $a$ and $b$ have disjoint support: for all $1 \le i \le n$, at most one of $a_i$ and $b_i$ are nonzero. (Otherwise, we risk that $a_i \pm b_i \not\in \{-1, 0, 1\}$.)

A monochromatic planar $K_4$ is obtained whenever, for two directions $a$ and $b$ with disjoint support, $a$, $b$, $a+b$, and $a-b$ are all the same color.

First consider only the directions in $\{0,1\}^n \subset \{-1,0,1\}^n$. By using Folkman's finite unions theorem (see, for example, p.\ 82 in \cite{graham1990ramsey}) we can choose four directions $a$, $b$, $c$, and $d$ among these with the following properties:
\begin{itemize}
\item $a$, $b$, $c$, and $d$ have disjoint support.
\item The 15 directions $a, \dots, d, a+b, a+c, \dots, c+d, a+b+c, \dots, b+c+d, a+b+c+d$ are the same color (say, red).
\end{itemize}

If $a-b$ is red, then $\{a, b, a+b, a-b\}$ determine a red planar $K_4$, so assume $a-b$ is blue. Similarly, if $c-d$ is red, then $\{c, d, c+d, c-d\}$ determine a red planar $K_4$, so assume $c-d$ is blue.

If $a-b+c-d$ is red, then $\{a+c, b+d, a+b+c+d, a-b+c-d\}$ determine a red planar $K_4$, so assume $a-b+c-d$ is blue. Similarly, if $a-b-c+d$ is red, then $\{a+d, b+c, a+b+c+d, a-b-c+d\}$ determine a red planar $K_4$, so assume $a-b-c+d$ is blue.

But now $\{a-b, c-d, a-b+c-d, a-b-c+d\}$ determine a blue planar $K_4$, and we have what we wanted.

It remains to check that the dimension $n$ required by the finite unions theorem in this case is not too large. We rely on the second proof outlined in \cite{graham1990ramsey}. 

Let $n(k)$ be the dimension needed to obtain $k$ directions $a^1, \dots, a^k$ with the following properties: for each nonempty $I \subseteq [k]$, the color $\sum_{i \in I} a^i$ is determined only by $\max\{I\}$. As a base case, $n(1) = 1$, since then any direction suffices.

To go from $n(k)$ to $n(k+1)$, let $n = \HJ(2,2, n(k))$ and choose a monochromatic $n(k)$-dimensional combinatorial subspace of $\{0,1\}^n$. This can be described by directions $b^0, \dots, b^{n(k)} \in \{0,1\}^n$ (with disjoint support) such that for all $I \subseteq [n(k)]$, $b^0 + \sum_{i\in I} b^I$ is the same color (say, red).

The set of all possible sums of $b^1, \dots, b^{n(k)}$ is isomorphic to $\{0,1\}^{n(k)}$, so we can find $k$ directions $a^1, \dots, a^k$, which are sums of some of the $b^i$ and have the property we want. Furthermore, let $a^{k+1} = b^0$. Then for all nonempty $I \subseteq [k+1]$, the sum $\sum_{i \in I} a^i$ is determined by $\max\{I\}$: this is true by the inductive hypothesis if $\max\{I\} \le k$, and if $\max\{I\} = k+1$, the sum lies in the combinatorial subspace we found, and is red. Therefore $n(k+1) \le n$.

By the bound in \cite{shelah1988primitive}, $\HJ(2,2,d) \le 2^{2^{2d}}$, so $n(k+1) \le 2^{2^{2n(k)}} \le 2^{2^{2^{n(k)}}}$. Since $n(1) = 1 = 2 \uparrow \uparrow 0$, $n(7) \le 2 \uparrow \uparrow 18$.

Finally, if we take $n = n(7)$, we can find seven directions $a^1, \dots, a^7$ as above. Choose four of these that are the same color; then because $\sum_{i \in I} a^i$ has the color of $a^{\max\{I\}}$, all their sums will share that color, and we can use them above to obtain a monochromatic planar $K_4$.
\end{proof}

\subsection{The upper bound}

\begin{proof}[Proof of Theorem~\ref{thm:upper}]

Let $n = \TTT(4,2,d)$, where $d$ is either $6$ or $2\uparrow \uparrow 18$, depending on whether Lemma~\ref{lem:computer} or Lemma~\ref{lem:uncomputer} is used. Let $Q$ be the cube $\{\pm1\}^{n+1}$. Given a $2$-coloring of the edges of $Q$, we consider just the edges from $Q^-$ to $Q^+$, and apply $\Phi$ to them to get a coloring of $[4]^n$. This coloring must contain a monochromatic $d$-dimensional tic-tac-toe space; by Lemma~\ref{lem:hyperbowtie}, its preimage in $Q$ is a $d$-dimensional monochromatic hyperbowtie.

From now on, we will look only at the $(d+1)$-dimensional subcube containing this hyperbowtie. What we know about this cube is that all edges which change the first coordinate (which we'll call the ``middle" of the cube) are colored the same color, which may as well be red. The remaining edges are contained in one of two $d$-dimensional cubes: the ``top" and ``bottom".

We reduce the problem of finding a monochromatic planar $K_4$ in this subcube to Lemma~\ref{lem:computer}. We color the edges of $\{\pm1\}^d$ as follows:
\begin{itemize}
\item[(1)] An equivalence class of parallel edges is colored blue, if the corresponding edges on the top are all colored blue.

\item[(2)] An equivalence class of parallel edges is colored red, if the corresponding edges on the bottom are all colored blue.

\item[(3)] If neither of these occurs, then there is a pair of parallel edges, one on the top and one on the bottom, which are colored red. Together with four edges in the middle, which are also red, they form a monochromatic planar $K_4$.
\end{itemize}

We are done if (3) holds for some equivalence class of parallel edges. Otherwise, by Lemma~\ref{lem:computer} or Lemma~\ref{lem:uncomputer}, the coloring we obtain contains a monochromatic planar $K_4$. If it is blue, then the corresponding $K_4$ on the top is monochromatic blue. If it is red, then the corresponding $K_4$ on the bottom is monochromatic blue.
\end{proof}

\section{Monochromatic planar squares}

For a vertex $v$ of the $n$-dimensional cube and $1 \le i \le n$, let $v\oplus i$ denote the vertex obtained by flipping the $i$-th coordinate of $v$. Whenever we refer to length or distance between two vertices, it will be Hamming distance: the number of coordinates in which the two coordinates differ.

\begin{lem}
\label{lem:sq-adj}
For $n\ge 4$, in any edge-coloring of the $n$-dimensional cube, at least $\frac12 - \frac1{2\floor{\frac n2} - 2}$ of all right angles formed by edges of length $2$ are monochromatic.
\end{lem}
\begin{proof}
Choose a vertex $v$ of the $n$-dimensional cube, and a permutation $\pi$ of $\{1, \dots, n\}$.

Let $k = \floor{\frac n2}$; for $1 \le i \le k$, let $w_i = v \oplus \pi(2i-1) \oplus \pi(2i)$. Then the edges $(v,w_1), \dots, (v,w_k)$ are mutually perpendicular and have length $2$. There are ${k \choose 2}$ pairs of edges in this set; the number of monochromatic pairs is minimized if $\frac k2$ of the edges are red, and $\frac k2$ are blue, for a total of $2 {k/2 \choose 2}$ monochromatic pairs.

When $k \ge 2$, the ratio of these is $\frac12 - \frac1{2(k-1)}$, which is the proportion of monochromatic pairs among these edges. By averaging over all choices of $v$ and $\pi$, we obtain the same proportion over the entire cube.
\end{proof}

\begin{lem}
\label{lem:sq-opp}
For $n \ge 4$, in any edge-coloring of the $n$-dimensional cube, at least $\frac12 - \frac1{2(n-3)}$ of all pairs of parallel edges of length $2$, which are also at distance $2$ from each other, are monochromatic.
\end{lem}
\begin{proof}
Choose a vertex $v$ of the $n$-dimensional cube, and a permutation $\pi$ of $\{1, \dots, n\}$.

Let $k=n-2$; for $1 \le i \le k$, let $v_i = v \oplus \pi(i)$ and $w_i = v_i \oplus \pi(n-1) \oplus \pi(n)$. Then the edges $(v_1,w_1), \dots, (v_k, w_k)$ are all parallel, have length $2$, and are at distance $2$ from each other. There are ${k \choose 2}$ pairs of edges in this set; the number of monochromatic pairs is minimized if $\frac k2$ of the edges are red, and $\frac k2$ are blue, for a total of $2 {k/2 \choose 2}$ monochromatic pairs.

When $k \ge 2$, the ratio of these is $\frac12 - \frac1{2(k-1)}$, which is the proportion of monochromatic pairs among these edges. By averaging over all choices of $v$ and $\pi$, we obtain the same proportion over the entire cube.
\end{proof}

\begin{lem}
\label{lem:sq-odd}
For $n \ge 5$, in any edge-coloring of the $n$-dimensional cube, at most $\frac{14}{15}$ of all $2\times 2$ squares have an odd number of red edges.
\end{lem}
\begin{proof}
Choose a vertex $v$ of the $n$-dimensional cube, and a permutation $\pi$ of $\{1, \dots, n\}$.

Let $v_1, \dots, v_{10}$ be $v \oplus \pi(i) \oplus \pi(j)$, for $1 \le i < j \le 5$. Using only edges of length $2$ between these vertices, $15$ squares can be formed, which together use each edge exactly twice.

Assume for the sake of contradiction that these edges are colored so that all $15$ squares have an odd number of red edges. Represent the two colors, red and blue, by $1$ and $0$, and let the sum of a square be the sum of the colors of its edges. The sums of all $15$ squares must be odd, so adding up all 15 sums, we also get an odd number. However, each edge is used twice and therefore contributes an even number to this total; a contradiction. 

Therefore at most $\frac{14}{15}$ of these squares have an odd number of red edges. By averaging over all choices of $v$ and $\pi$, we obtain the same proportion over the entire cube.
\end{proof}

\begin{proof}[Proof of Theorem~\ref{thm:squares}] There are four types of colorings of $2 \times 2$ squares, up to symmetry and interchanging the two colors:
$$\img{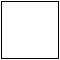} \quad \img{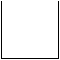} \quad \img{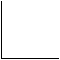} \quad \img{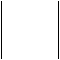}$$

For $n \ge 5$, fix an edge-coloring of the $n$-dimensional cube. We will use the four symbols above to denote the proportions of $2\times 2$ squares of each type. By applying the lemmas, we can write the following system of inequalities:
\begin{align}
\img{sq4.pdf} + \img{sq3.pdf} + \img{sq2adj.pdf} + \img{sq2opp.pdf} &= 1 \\
\img{sq4.pdf} + \frac12 \cdot \img{sq3.pdf} + \frac12 \cdot \img{sq2adj.pdf} &\ge \frac12 - \frac1{2\floor{\frac n2} - 2} \\
\img{sq4.pdf} + \frac12 \cdot \img{sq3.pdf} + \img{sq2opp.pdf} &\ge \frac12 - \frac1{2(n-3)} \\
\img{sq3.pdf} &\le \frac{14}{15}.
\end{align}
Solving for \img{sq4.pdf} by taking $2 \cdot (2) + (3) - (1) - \frac12 \cdot (4)$, we get
$$2 \cdot \img{sq4.pdf} \ge \frac{1}{30} - \frac1{\floor{\frac n2} - 1} - \frac1{2(n-3)}.$$
For $n \ge 78$, the left-hand side is positive, and therefore a monochromatic $2 \times 2$ square exists.
\end{proof}

\appendix
\section{Computer Proof of Lemma~\ref{lem:computer}}

The following Mathematica code, which was used to verify Lemma~\ref{lem:computer}, will also be made available at \url{http://www.math.cmu.edu/~mlavrov/other/Graham.nb} for as long as possible. The online version includes additional code which, for $n \le 5$, will draw an edge coloring containing no monochromatic planar $K_4$.

\lstset{language=[5.2]Mathematica}
\subsection{Subroutines} Here we define several short subroutines that will be used in the next section to construct the 6-dimensional cube.

\begin{lstlisting}
(* CanonicalForm makes a canonical choice of sign
   for a vector *)
CanonicalForm[edge_] := Last[Sort[{edge, -edge}]]


(* DisjointSupportQ returns True if two vectors have
   disjoint support: in each coordinate, at most one 
   can be nonzero.
   Equivalently, the edges are perpendicular. *)
DisjointSupportQ[{edge1_, edge2_}] :=
Module[{support1, support2},
   support1 = Flatten[Position[edge1, 1 | -1]]; 
      (* which elements of edge1 are nonzero *)
   support2 = Flatten[Position[edge2, 1 | -1]]; 
      (* which elements of edge2 are nonzero *)
   Return[Intersection[support1, support2] == {}];
];


(* MakeK4FromRectangle takes two edges defining a
   rectangle as input. It returns the edges of the
   unique planar K4 containing that rectangle. *)
MakeK4FromRectangle[{a_, b_}] := 
   {a, b, CanonicalForm[a+b], CanonicalForm[a-b]};


(* NotMonochromatic defines a Boolean formula on the
   variables corresponding to some edges, which is
   True iff the edges are not all the same color. *)
NotMonochromatic[{a_, b_, c_, d_}] := 
   (x[a]~Or~x[b]~Or~x[c]~Or~x[d])
      ~And~
   Not[x[a]~And~x[b]~And~x[c]~And~x[d]]
\end{lstlisting}
\subsection{Structure} Next, we initialize variables that store the structure of the cube: the parallel edge classes and the planar $K_4$s they form.

\begin{lstlisting}
n = 6; 
   (* the dimension of the cube *)
edges = Tuples[{-1, 0, 1}, n]; 
   (* parallel edge classes are defined by whether they
      increase (+1), decrease (1) or don't change (0)
      on each coordinate *)
edges = Cases[edges, Except[{0 ..}]];
   (* we exclude the trivial edge which does not change
      any coordinates *)
edges = DeleteDuplicates[Map[CanonicalForm, edges]]; 
   (* we put each edge in canonical form, removing
      duplicates *)

rectangles = Select[Subsets[edges, {2}], DisjointSupportQ];
   (* two edges form the sides of a rectangle if they 
      change disjoint sets of coordinates *)

k4s = Map[MakeK4FromRectangle, rectangles];
   (* all planar K4s are made by the sides and diagonals
      of a rectangle *)
\end{lstlisting}

\subsection{Satisfiability} Finally, we express the coloring problem as a SAT instance, which allows us to use Mathematica's built-in commands to check that no solution exists.

\begin{lstlisting}
variables = Map[x, edges]; 
   (* for each edge we have a binary variable, True or
      False depending on the color of the edge *)
constraints = Map[NotMonochromatic, k4s]; 
   (* the constraints are simply that no K4 is 
      monochromatic *)
formula = Apply[And, constraints];

output = SatisfiabilityInstances[formula, variables] 
   (* this will attempt to find an assignment of 
      True/False values that will satisfy all constraints.
      An output of {} means that this is impossible. *)
\end{lstlisting}

Expect the code in this section to take several minutes to run.

\section{Rate of growth of Shelah's Hales-Jewett bounds}

We will make use of Knuth's up-arrow notation. Let $a \uparrow b := a^b$; then define
$$a \uparrow \uparrow b = \underbrace{a \uparrow (a \uparrow (a \uparrow \cdots \uparrow a))}_{b\; a\mathrm{'s}}.$$
Finally, define
$$a \uparrow \uparrow \uparrow  b =  \underbrace{a \uparrow \uparrow  (a \uparrow \uparrow  (a \uparrow \uparrow  \cdots \uparrow \uparrow  a))}_{b\; a\mathrm{'s}}.$$
We will use the following rules to rewrite expressions written in this notation:
\begin{itemize}
\item An expression of the form $a \uparrow a \uparrow \cdots \uparrow a$ with arbitrarily-inserted parentheses is always maximized when the parentheses are placed as in the definition of $a \uparrow\uparrow b$.

\item Therefore $(a \uparrow \uparrow b) \uparrow \uparrow c \le a \uparrow \uparrow (b\cdot c)$, by expanding and rearranging the parentheses.

\item Since $(2 \uparrow \uparrow k)^2 \le 2\uparrow \uparrow (k+1)$,  it is also true that $a \cdot (2\uparrow \uparrow k) < 2\uparrow \uparrow (k+1)$ for any $a < 2\uparrow \uparrow k$.

\item Finally, $a + (2 \uparrow \uparrow k) < 2 \cdot (2 \uparrow \uparrow k) \le 2 \uparrow \uparrow (k+1)$, for any $a < 2 \uparrow \uparrow k$.
\end{itemize}

\begin{lem}
\label{lem:up-arrow}
If $\HJ(t-1, 2, d) \le 2 \uparrow \uparrow m$, then $\HJ(t,2,d) \le 2 \uparrow\uparrow 2 \uparrow\uparrow(m+3)$.
\end{lem}
\begin{proof}
From \cite{shelah1988primitive}, we know the following: suppose $\HJ(t-1, k, d) = n$. Then $\HJ(t, k, d) \le n f(n, k^{t^n})$, where $f(\ell,k)$ is defined recursively by $f(1,k)=k+1$ and $f(\ell+1,k) = k^{f(\ell,k)^{2\ell}}+1$.

We begin by bounding $f(\ell,k)$ in up-arrow notation. Whenever we will need to find $f(\ell,k)$, we will have $k>2\ell$. Thus, we can write $f(\ell,k) < k^{f(\ell-1,k)^k}$; iterating this bound, and rearranging the parentheses, we get
$$f(\ell,k) < \underbrace{k^{k^{k^{\iddots^k}}}}_{2\ell} = k \uparrow \uparrow 2\ell.$$

Suppose $\HJ(t-1,2,d) \le 2 \uparrow\uparrow m$. It's easy to check that $\HJ(t-1,2,d) \ge t-1$, and therefore $t \le 1 + 2\uparrow\uparrow m \le 2\uparrow\uparrow (m+1).$ So we have $2^{t^n} \le 2 \uparrow\uparrow (2m+2)$. Therefore
\begin{align*}
n f(n, 2^{t^n}) &\le (2\uparrow\uparrow m) \cdot f( 2\uparrow\uparrow m, 2 \uparrow\uparrow (2m+2)) \\
&\le (2 \uparrow \uparrow m) \cdot [(2 \uparrow \uparrow (2m+2)) \uparrow \uparrow (2 \cdot 2 \uparrow \uparrow m)]\\
&\le (2 \uparrow \uparrow m) \cdot [(2 \uparrow \uparrow (2m+2)) \uparrow \uparrow (2 \uparrow \uparrow (m+1))]\\
&\le (2 \uparrow \uparrow m) \cdot [2 \uparrow \uparrow ((2m+2) \cdot 2 \uparrow \uparrow (m+1))] \\
&\le (2 \uparrow \uparrow m) \cdot [2 \uparrow \uparrow (2 \uparrow \uparrow (m+2))] \\
&\le 2 \uparrow \uparrow (1 + 2 \uparrow \uparrow (m+2)) \\
&\le 2 \uparrow \uparrow 2 \uparrow \uparrow (m+3). \qedhere
\end{align*}
\end{proof}

\begin{lem}
\label{lem:426}
$\HJ(4,2,6) < 2 \uparrow \uparrow 2 \uparrow \uparrow 2 \uparrow \uparrow 9 < 2 \uparrow \uparrow \uparrow 6$.
\end{lem}
\begin{proof}
The doubly exponential bound on $\HJ(2,2,6)$ from \cite{shelah1988primitive} yields $\HJ(2,2,6) \le 2^{2^{12}} < 2\uparrow \uparrow 5$. Applying Lemma~\ref{lem:up-arrow}, we get $\HJ(3,2,6) < 2 \uparrow \uparrow 2 \uparrow \uparrow 8$, and
\begin{align*}
\HJ(4,2,6) &< 2 \uparrow \uparrow 2 \uparrow \uparrow (3 + 2 \uparrow \uparrow 8) \\
&< 2 \uparrow \uparrow 2 \uparrow \uparrow 2 \uparrow \uparrow 9 \\
&< 2 \uparrow \uparrow 2 \uparrow \uparrow 2 \uparrow \uparrow 65536 \\
&= 2 \uparrow \uparrow 2 \uparrow \uparrow 2 \uparrow \uparrow 2 \uparrow \uparrow 2 \uparrow \uparrow 2 \\
&= 2\uparrow \uparrow \uparrow 6. \qedhere
\end{align*}

With $d = 2\uparrow \uparrow 18$ in place of $6$, we obtain $2 \uparrow \uparrow 2 \uparrow \uparrow 2 \uparrow \uparrow 25$ by exactly the same reasoning.

\end{proof}

\end{document}